\pdfoutput=1
\RequirePackage{ifpdf}
\ifpdf 
\documentclass[pdftex]{sigma}
\else
\documentclass{sigma}
\fi

\usepackage{tikz}

\newcommand{\sA}{{\mathcal A}}
\newcommand{\sC}{{\mathcal C}}
\newcommand{\sI}{{\mathcal I}}
\newcommand{\sO}{{\mathcal O}}
\newcommand{\sP}{{\mathcal P}}
\newcommand{\sS}{{\mathcal S}}
\newcommand{\C}{{\mathbb C}}
\renewcommand{\H}{{\mathbb H}}
\renewcommand{\P}{{\mathbb P}}
\newcommand{\Q}{{\mathbb Q}}
\newcommand{\surj}{\twoheadrightarrow}

\numberwithin{equation}{section}

\newtheorem{thm}{Theorem}[section]
\newtheorem{lem}[thm]{Lemma}
\newtheorem{prop}[thm]{Proposition}
{ \theoremstyle{definition}
\newtheorem{Remark}[thm]{Remark} }

\begin{document}

\newcommand{\arXivNumber}{2105.????}

\renewcommand{\thefootnote}{}

\renewcommand{\PaperNumber}{048}

\FirstPageHeading

\ShortArticleName{Double Box Motive}

 \ArticleName{Double Box Motive\footnote{This paper is a~contribution to the Special Issue on Algebraic Structures in Perturbative Quantum Field Theory in honor of Dirk Kreimer for his 60th birthday. The~full collection is available at \href{https://www.emis.de/journals/SIGMA/Kreimer.html}{https://www.emis.de/journals/SIGMA/Kreimer.html}}}

\Author{Spencer BLOCH}

\AuthorNameForHeading{S.~Bloch}

\Address{Department of Mathematics, The University of Chicago,\\ Eckhart Hall, 5734 S University Ave, Chicago IL, 60637, USA}
\Email{\href{mailto:spencer_bloch@yahoo.com}{spencer\_bloch@yahoo.com}}
\URLaddress{\url{http://www.math.uchicago.edu/~bloch/}}

\ArticleDates{Received March 20, 2021, in final form May 04, 2021; Published online May 13, 2021}

\Abstract{The motive associated to the second Symanzik polynomial of the double-box two-loop Feynman graph with generic masses and momenta is shown to be an elliptic curve.}

\Keywords{Feynman amplitude; elliptic curve; double-box graph; cubic hypersurface}

\Classification{14C30; 14D07; 32G20}

\renewcommand{\thefootnote}{\arabic{footnote}}
\setcounter{footnote}{0}

\section{Introduction}\label{S1}

Cubic hypersurfaces in algebraic geometry are sort of~analogous to the baby bear's porridge in~the tale of~Goldilocks and the three bears. Hypersurfaces of~degrees one and two are too cold and uninteresting, while degrees four and higher are too hot to handle. Degree three is just right. The author thanks Dirk Kreimer and Pierre Vanhove for explaining the importance for physics of~two-loop Feynman graphs whose amplitudes are (mixed, or relative) periods of~certain singular cubic hypersurfaces. In particular, the paper~\cite{K} suggests the central role the singularities of~the second Symanzik hypersurface play. The basic setup of~iterated subdivision of~the two-loop sunset grew out of~collaboration with Vanhove~\cite{V}. The kite graph (written $(2,1,2)$ in the notation explained in Appendix~\ref{app}) also leads to an elliptic curve. The situation for the kite is more elementary since one can fall back on the classical theory of~cubic threefolds with isolated double points in~$\P^4$. Vanhove has been able to calculate the $j$-invariant for the resulting elliptic curve. A similar calculation for the double box looks difficult. The author was also influenced by unpublished work of Matt Kerr classifying Feynman motives associated to a~number of~other two-loop diagrams.

This paper focuses on what is perhaps the most striking example, the massive second Syman\-zik motive associated to the double box graph, see Figure~\ref{fig1}.
The massive second Symanzik in our case is a~(singular) cubic hypersurface $X_{3,1,3}$ in $\P^6$. For a smooth cubic hypersurface $X\subset \P^6$, the Hodge structure is the Hodge structure associated with the cohomology in middle dimension, in~this case $H^5(X,\Q)$. It is known \cite[p.~16]{H}, that $H^5(X,\Q)\cong H^1(A,\Q(-2))$ where $A$ is a certain abelian variety of dimension $21$. For the double box hypersurface with generic momenta and masses, $H^5(X_{3,1,3},\Q)$ is a mixed Hodge structure with weights $\le 5$. We will construct a specific resolution of singularities $\pi\colon Z \to X_{3,1,3}$. We will show $F^3H^5(Z) =\C$ and $F^4H^5(Z) =(0)$. From this it will follow that as a Hodge structure, $H^5(Z,\Q) \cong H^1(E,\Q(-2))$ for a suitable elliptic curve $E$. (We do not specify the elliptic curve.) The dimension of the abelian variety drops from $21$ to $1$!

\begin{figure}[ht]\centering
\begin{tikzpicture}
\draw (-3,1) -- (3,1) -- (3,-1) -- (-3,-1) -- (-3,1);
\draw (0,2) -- (0,-2);
\draw (-3.75,1.75) -- (-3,1);
\draw (-3.75,-1.75) -- (-3,-1);
\draw (3.75,-1.75) -- (3,-1);
\draw (3.75,1.75) -- (3,1);
\node at (-1.5,-0.75) {\small 1};
\node at (-1.5,0.75) {\small 3};
\node at (1.5,-0.75) {\small 7};
\node at (-0.15,0) {\small 4};
\node at (1.5,0.75) {\small 5};
\node at (2.85,0) {\small 6};
\node at (-2.85,0) {\small 2};
\end{tikzpicture}
\\[2ex]
\begin{tikzpicture}
\draw (0,1.5) -- (0,-1.5);
\draw (0,1) .. controls (0.75,0.75) and (0.75,-0.75) .. (0,-1);
\draw (0,1) .. controls (-0.75,0.75) and (-0.75,-0.75) .. (0,-1);
\node at (-0.5,-0.85) {\small 1};
\node at (-0.5,0.85) {\small 3};
\node at (0.5,-0.85) {\small 7};
\node at (0.15,0) {\small 4};
\node at (0.5,0.85) {\small 5};
\node at (0.75,0) {\small 6};
\node at (-0.75,0) {\small 2};
\filldraw (0,1) circle (0.04);
\filldraw (0,-1) circle (0.04);
\filldraw (-0.5,-0.4) circle (0.04);
\filldraw (0.5,-0.4) circle (0.04);
\filldraw (0.5,0.4) circle (0.04);
\filldraw (-0.5,0.4) circle (0.04);
\draw (-0.5,-0.4) -- (-1,-0.65);
\draw (0.5,-0.4) -- (1,-0.65);
\draw (0.5,0.4) -- (1,0.65);
\draw (-0.5,0.4) -- (-1,0.65);
\node at (2,0) {\small $=(3,1,3)$};
\node at (-3.0,0.2) {\small $\leftrightarrow$ subdivided};
\node at (-2.6,-0.2) {\small sunset graph:};
\filldraw[transparent] (4.2,0) circle (0.04);
\end{tikzpicture}
\caption{Double box graph.}\label{fig1}
\end{figure}
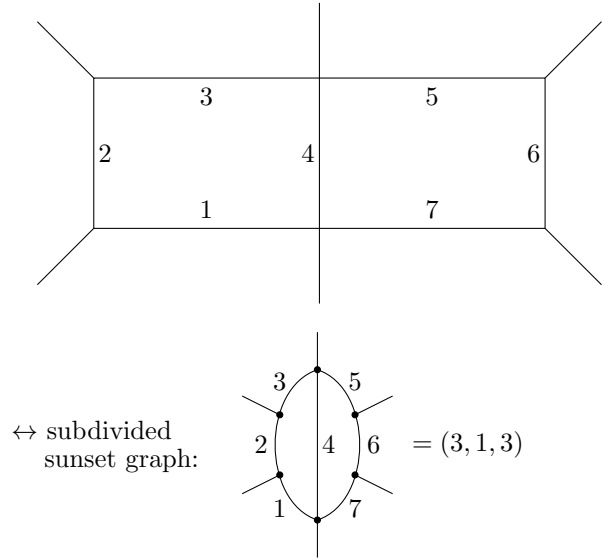

{\samepage From the structure of the map $\pi$ we will deduce our main result
\begin{thm}
With the above notation,
\begin{gather*}
H^1(E,\Q(-2)) \cong {\rm gr}^W_5H^5(X_{3,1,3},\Q).
\end{gather*}
\end{thm}}

The proof is given in three steps.
\begin{enumerate}\itemsep=0pt
\setlength{\leftskip}{0.45cm}
\item[\textit{Step} 1.] We construct the resolution of singularities $Z \to X_{3,1,3}$ and we show $F^3H^5(Z,\C)\cong \C$. This is Theorem~\ref{mainthm} below.

\item[\textit{Step} 2.] We show $F^4H^5(Z,\C) = (0)$, so the Hodge type of $H^5(Z)$ is $(0,0,1,1,0,0)$, from which it follows that $H^5(Z)\cong H^1(E,\Q(-2))$ for a suitable elliptic curve.

\item[\textit{Step} 3.] We show $H^5(Z,\Q)\cong {\rm gr}^W_5H^5(X_{3,1,3},\Q)$.
\end{enumerate}

The study of motives associated to two-loop graphs involves a lot of detailed algebraic geometry. The author can only hope he has gotten the details right! Another approach, which might yield more information in this case about the elliptic curve, could be based on the study of~linear spaces contained in the second Symanzik variety. For example, projecting from a double point on the variety realizes a singular cubic hypersurface of dimension $n$ as (birational to) the blowup of an intersection of a quadric and a cubic hypersurface on~$\P^n$~\cite{C}. When $n=5$, this intersection will be a (singular) Fano threefold, and one may expect a rough dictionary
$$\begin{array}{c}
\text{intermediate jacobian of Fano} \leftrightarrow \text{lines on Fano} \leftrightarrow
\text{planes on cubic fivefold}
\\[.5ex]
\leftrightarrow \text{intermediate jacobian of cubic fivefold.}
\end{array}$$
The weight $5$ piece of $H^5$ for the double box $X$ (with generic parameters) is seen to have Hodge type $(2,3)$, $(3,2)$ which suggests the role of~$\P^2$'s. Namely, we might expect a correspondence of~the form
\begin{gather*}
H^1(C) \xrightarrow{\alpha^*} H^1(P) \xrightarrow{\beta_*} \big(W_5H^5(X)\big)(2).
\end{gather*}
Here $P \xrightarrow{\alpha} C$ is a family of $\P^2$'s parametrized by a curve $C$. A deeper understanding of the elliptic curve arising from $H^5$ of the double box~\eqref{1} should come from a study of the Fano variety of planes on~\eqref{1}. (Presumably this Fano is a surface, see \cite[Section~1]{C}.) This is a~wonderful problem! In the case of the kite graph $(2,1,2)$ one is able to write the elliptic curve as an intersection of two quadrics in $\P^3$. I do not know if similar methods will work for the double box. Another approach, suggested by the referee, would be to study the Picard--Fuchs equation associated to $H^5$ of the double box~\eqref{1}, relating it if possible to the Picard--Fuchs equation for~a~family of elliptic curves.

More generally it would be interesting to generalize motivic considerations to cover all $2$-loop graphs with generic parameters. A physicist in the audience at the Newton institute suggested that the non-planar case would be more subtle. I asked him to email me so we could correspond, but he never did, so I cannot give him credit, but he is certainly correct. A calculation for the graph $(3,2,2)$ shows that $F^3H^5$ (resolution of $(3,2,2)$) has dimension $\ge 5$. On the other hand, the graphs $(n,1,n)$ seem better behaved.

\section[Singularities of X\{3,1,3\}]{Singularities of $\boldsymbol{X_{3,1,3}}$}
We focus now on singularities for the double box (for details, see Proposition~\ref{2symm1n})
\begin{gather}
X_{3,1,3}\colon\ 0=\Psi_{3,1,3}=Q(x_5,x_6,x_7)(x_1+x_2+x_3+x_4)\nonumber
\\ \hphantom{X_{3,1,3}\colon\ 0=\Psi_{3,1,3}=}
{}+Q'(x_1,x_2.x_3)(x_4+x_5+x_6+x_7) +x_4\sA. \label{1}
\end{gather}

Here $\sA$ has degree $2$ and is linear separately in each variable $x_1,\dots,x_7$.
\begin{prop}
The singular locus $X_{3,1,3,{\rm sing}}= C\amalg C'\amalg \sS$, where
\begin{gather*}
C\colon\ Q(x_5,\dots,x_7) = 0 = x_{1}=x_{2}=\cdots = x_{4}, \nonumber
\\
C'\colon\ Q'(x_{1},\dots,x_{3}) = 0 = x_{4}=x_{5}=\cdots = x_{7}.
\end{gather*}
Here $C$ $($resp.~$C')$ is a smooth quadric in $\P^{2}$. Finally $\sS$ is a finite set of ordinary double points on $X_{3,1,3}$ defined by
\begin{gather*}
Q(x_5,x_6,x_7)=Q'(x_1,x_2,x_3)=x_4=x_1+x_2+x_3=x_5+x_6+x_7=\sA=0.
\end{gather*}
The double points $\sS$ are disjoint from $C$ and $C'$.
\end{prop}
\begin{proof} The reader can easily check that points of $C$, $C'$, and $\sS$ are singular on $X_{3,1,3}$. We~can understand points of $\sS$ by first imposing the linear relations $x_4=x_1+x_2+x_3=x_5+x_6$ $+\,x_7=0$,~so
\begin{gather*}
\sS\colon\ Q(x_5,x_6,-x_5-x_6)=Q'(x_1,x_2,-x_1-x_2)
\\ \hphantom{\sS\colon\ Q(x_5,x_6,-x_5-x_6)}
{}=\sA(x_1,x_2,-x_1-x_2,0,x_5,x_6,-x_5-x_6)=0
\end{gather*}
in $\P^3_{x_1,x_2,x_5,x_6}$.
\end{proof}
\begin{Remark}It will be important in the sequel that points in $\sS$ all lie on the hyperplane $x_4=0$.
\end{Remark}
\section{Blowings up}
We blow up $C\amalg C'$ and $\sS$ on $\P^6$ and on $X_{3,1,3}$, obtaining the diagram
\begin{equation*}\qquad\qquad\qquad\qquad\qquad\qquad
\begin{CD}Z @>>> W \\
@VVV @VV\sigma=\text{blow }\sS V \\
Y @>>> V \\
@VVV @VV \pi=\text{blow $C\amalg C'$} V \\
X_{3,1,3} @>>> \P^6
\end{CD}
\end{equation*}
The exceptional divisor in $V$ is written $E\amalg E'$ with
\begin{gather*}
E = \P\big(N^\vee_{C\subset \P^6}\big),\qquad E' = \P\big(N^\vee_{C'\subset \P^6}\big).
\end{gather*}
Both conormal bundles have the form $N^\vee = \sO(-1)^4\oplus \sO(-2)$. Since $\sS$ is disjoint from $E\amalg E'$, the map $\pi\circ\sigma$ is itself a blowup, with exceptional fibre
\begin{gather*}
\amalg_6 \P^5\amalg E\amalg E'.
\end{gather*}
\begin{lem}$Z$ is a smooth divisor on $W$.
\end{lem}
\begin{proof}The points of $\sS$ are ordinary double points on $X_{3,1,3}$. Since the fibres of $Z$ over these points are smooth quadrics, it suffices to show $Y\cap (E\amalg E')$ is smooth. The situation is symmetric so we can focus on $Y\cap E$. We have $C\colon x_1=x_2=x_3=x_4=Q(x_5,x_6,x_7)=0$. Let $I\subset \sO_{\P^6}$ be the ideal of $C$. Since $\Psi_{3,1,3}$ vanishes to order $2$ along $C$, $\Psi$ defines a map $\sO_{\P^6}(-3) \to I^2$. The image of this map defines a homogeneous sheaf of ideals of degree $2$ for the graded sheaf of~algebras~$\bigoplus I^k/I^{k+1}$ with corresponding variety $E\cap Y$.

To avoid having to keep track of the grading on $I/I^2$, write $y_i=x_i/x_5$. Write $q = Q(x_5,x_6,x_7)/x_5^2$ for the coordinate in the fibre. In contrast $q'(y_1,y_2,y_3)=Q'(x_1,x_2,x_3)/x_5^2$. The defining equation locally looks like
\begin{gather}\label{6}
q(y_1+y_2+y_3+y_4)+q'(y_1,y_2,y_3)(1+y_6+y_7)+y_4L(y_1,y_2,y_3,y_4).
\end{gather}
Here $L$ is defined (compare the definition of $\sA$ in~\eqref{11a}).
\begin{gather*}
L(y_1,y_2,y_3,y_4) = (y_1,y_2,y_3,y_4)
\begin{pmatrix}a_{15}&a_{16}&a_{17} \\a_{25}&a_{26}&a_{27} \\
a_{35}&a_{36}& a_{37}\\ a_{45}&a_{46}& a_{47} \end{pmatrix}
\begin{pmatrix}1\\y_6 \\y_7\end{pmatrix}\!.
\end{gather*}
Note that~\eqref{6} does not vanish to order $\ge 2$ at any point of $E\cap Y$. (It may clarify to remark that $E\cap Y$ is not smooth over $C$. Indeed, at the point of $C$ defined by $1+y_6+y_7=0$, the equation in the fibre becomes $q(y_1+y_2+y_3+y_4)+y_4L(y_1,y_2,y_3,y_4)=0$ and this can vanish to order $2$.)
\end{proof}
Smoothness of $Z$ permits us to calculate the Hodge filtration on $H^6(W-Z)$ using the pole order filtration \cite[Chapter~II, Proposition~3.13]{D}. $W$ is obtained by blowing up points and smooth rational curves on $\P^6$, so all cohomology classes on $W$ are Hodge. In particular the Gysin sequence yields
\begin{gather}\label{8}
F^3H^5(Z) \cong F^4H^6(W-Z) \cong \H^2\big(W,\Omega^4_W(Z) \xrightarrow{d} \Omega^5_W(2Z) \xrightarrow{d}\Omega^6_W(3Z)\big).
\end{gather}

\section{Leray}
To evaluate $F^3H^5(Z)$ in~\eqref{8}, we will use the Leray spectral sequence for $\pi\colon W \to \P^6$.
\begin{prop}\label{vanishing}
We have $R^i\pi_*\Omega^j_W(kZ)=(0)$, $i, k\ge 1$, $j\ge 0$.
\end{prop}
\begin{proof}For $i\ge 1$ it is clear that the support of $R^i\pi_*\Omega^j_W(kZ)$ is contained in $C\amalg C'\amalg \sS$.
\begin{lem} For $i\ge 1$, $R^i\pi_*\Omega^j_W(kZ)$ is killed by the ideal $J\subset \sO_{\P^6}$ of functions vanishing on~$C\amalg C'\amalg \sS$.
\end{lem}
\begin{proof}[Proof of lemma]
As a consequence of Zariski's formal function theorem,
\begin{gather*}
R^i\pi_*\Omega^j_W(kZ)^\wedge \cong \varprojlim_n R^i\pi_*\big(\Omega^j_W(kZ)\otimes \sO_W/J^n\big).
\end{gather*}
Since completion is faithfully flat, it will suffice to show the individual sheaves $R^i\pi_*\big(\Omega^j_W(kZ)\otimes \sO_W/J^n\big)$ are killed by $J$. For this it will suffice to show
\begin{gather*}
R^i\pi_*\big(\Omega^j_W(kZ)\otimes J^{n-1}/J^n\big) = (0),\qquad i\ge 1,\qquad n\ge 2,\qquad j,k\ge 0.
\end{gather*}
Since $X_{3,1,3}$ vanishes to order $2$ along $C\amalg C'\amalg \sS$, it follows that locally, the divisor $Z\cdot{}$excep\-tional divisor corresponds to the line bundle $\sO_{\P(N^\vee)}(2)$ where $N^\vee$ is the conormal bundle of $C\amalg C'\amalg \sS$. Thus we need to show
\begin{gather*}
R^i\pi_*\big(\Omega_V^j\otimes \sO_{\P(N^\vee)}(2k+n-1)\big)=(0).
\end{gather*}
For $j=0$ this is standard because we are computing cohomology in degree $>0$ of a positive multiple of the tautological bundle on projective space. For $j>0$, we have exact sequences (writing $E$ for the exceptional divisor)
\begin{gather}
0 \to \sO_E(-E) \to \Omega^1_W|E \to \Omega^1_E \to 0,\nonumber
\\
0 \to \Omega^{j-1}_{\P(N^\vee)}(1) \to \Omega^j_W|E \to \Omega^j_{\P(N^\vee)} \to 0. \label{12}
\end{gather}
It therefore suffices to show $R^i\pi_*\big(\Omega^\ell_{\P(N^\vee)}(m)\big)=(0)$ for $i, m\ge 1$, $\ell > 0$. We have a filtration on~$\Omega^\ell_{\P(N^\vee)}$ locally over the base with graded pieces $\Omega^q_{\P(N^\vee)/C\amalg C'\amalg \sS}$, $q=\ell-1, \ell$. The assertion is local on the base, so finally we have to check that $H^i\big(\P^3, \Omega^q_{\P^3}(m)\big)=(0)$ for $i, m\ge 1$. This is standard, because $\Omega_{\P^s}^j$ admits a resolution by direct sums of line bundles $\sO(-r)$ with $r\le s+1$.
\end{proof}
We return now to the proof of Proposition~\ref{vanishing}. Since for $i\ge 1$, $R^i\pi_*\big(\Omega^j_W(kZ)\big)$ is supported on~$C\amalg C'\amalg \sS$, it follows from the lemma that
\begin{gather*}
R^i\pi_*\big(\Omega^j_W(kZ)\big) \cong R^i\big(\pi|_{\pi^{-1}(C\amalg C'\amalg \sS)}\big)_*\big(\Omega^j_W(kZ)\big)|_{(\pi^{-1}(C\amalg C'\amalg \sS)}.
\end{gather*}
Finally, it follows from~\eqref{12} that $R^i\pi_*(\Omega^j_W(kZ)) = (0)$ for $k\ge 1$, proving Proposition~\ref{vanishing}.
\end{proof}

In order to evaluate~\eqref{8}, we need finally to calculate $\pi_*\Omega^{3+m}_W(mZ)$ for $m=1,2,3$. Note that sections of $\pi_*\Omega^{3+m}_W(mZ)$ coincide with sections $\omega$ of $\Omega^{3+m}_{\P^6}(mX)$ such that $\pi^*(\omega)$ has no pole along the exceptional divisor. First consider what happens over the ordinary double points~$\sS$. Here $\Psi_{3,1,3}$ vanishes to order $2$ so a pullback of a section of $\Omega^{3+m}_{\P^6}(mX)$ gets a pole of order~$2m$ along the exceptional divisor coming from the pullback of $mX$. On the other hand if the singular point is defined by $z_1=\cdots =z_6=0$, then writing ${\rm d}z_j = {\rm d}(z_i(z_j/z_i))$ it follows that any form ${\rm d}z_{i_1}\wedge\cdots\wedge {\rm d}z_{i_{3+m}}$ downstairs pulls back to a form vanishing to at least order $2+m$ along the exceptional divisor upstairs. For $m=1,2$ the pole order $2m-(2+m)$ upstairs is $\le 0$ so there is no pole. On the other hand, for $m=3$ we get a pole of order $1$ along the exceptional divisor. Let $\sI(\sS)\subset \sO_{\P^6}$ be the ideal of functions vanishing on $\sS$. Then
\begin{gather*}
\sI(\sS)\Omega^6_{\P^6}(3X)
\end{gather*}
have no poles on the exceptional divisor lying over $\sS$.

Over the curves $C$ and $C'$ local defining equations for the singular set have the form $z_1=\cdots =z_5=0$. A complete set of parameters locally on $\P^6$ becomes $z_1,\dots, z_5,\, t$ for some $t$. Locally, a $4$-form downstairs pulls back to have a zero of order at least $2$ on the exceptional divisor upstairs, from which it follows that
\begin{gather*}
\pi_*\Omega^4_W(Z) = \Omega^4_{\P^6}(X).
\end{gather*}
Sections of $\Omega^4_{\P^6}$ pull back to vanish to order $4$ on the exceptional divisor, so writing $I\subset \sO_{\P^6}$ for the ideal of functions vanishing on $C\amalg C'$, we have
\begin{gather*}
\pi_*\Omega^6_W(3Z) = (\sI(\sS))\cap I^2\Omega^6_{\P^6}(3X).
\end{gather*}

Finally, over $C\amalg C'$, the situation for $\pi_*\Omega^5_W(2X)$ is more complicated. Let $\sI\subset \sO_{\P^6}$ be the ideal of $C\amalg C'$. Write
\begin{gather*}
\Omega^1_{\P^6,\sI}:=\sI\Omega^1_{\P^6}+d\sI = \ker\big(\Omega^1_{\P^6} \to \Omega^1_{C\amalg C'}\big)
\end{gather*}
We have
\begin{gather*}
\pi_*\Omega^5_W(2Z) = \Omega^5_{\P^6,\sI}(2X):=\operatorname{Image}\bigg(\bigwedge^5 \Omega^1_{\P^6, \sI} \to \Omega^5_{\P^6}\bigg)(2X).
\end{gather*}
Since cohomology in degree $>0$ of $\Omega^i_{\P^n}(j)$ vanishes for $j>0$, we conclude in~\eqref{8} that
\begin{gather*}
F^3H^5(Z) \cong
\H^2\big(W, \Omega^4_{\P^6}(X)\to \Omega^5_{\P^6,\sI}(2X)\to \big(\sI(\sS)\cap I^2\big)\Omega^6_{\P^6}(3X)\big) \\ \hphantom{F^3H^5(Z)}
{}\cong\H^1\big(\P^6, \Omega^5_{\P^6,\sI}(2X) \to \big(\sI(\sS)\cap I^2\big)\Omega^6_{\P^6}(3X)\big).
\end{gather*}

\section[The computation of F3H5(Z)]{The computation of $\boldsymbol{F^3H^5(Z)}$}
\begin{thm}
\label{mainthm} With notation as above, we have $F^3H^5(Z) \cong \C$.
\end{thm}

\begin{proof}
\begin{lem}\label{firstlem}
$H^0\big( \Omega^5_{\P^6,\sI}(2X)\big)=(0)$.
\end{lem}
\begin{proof}[Proof of lemma] We have $\Omega^5_{\P^6}(2X) \cong T_{\P^6}(-1)$. Also $\sI\Omega^5_{\P^6}\subset \Omega^5_{\P^6,\sI}(2X)$, so
\begin{gather*}
T_{\P^6}|_{C\amalg C'}(-1) \surj \Omega^5_{\P^6}(2X)/\Omega^5_{\P^6,\sI}(2X).
\end{gather*}
Locally $\Omega^5_{\P^6,\sI}$ is generated by $\sI\Omega^5_{\P^6}$ and ${\rm d}z_1\wedge\cdots \wedge {\rm d}z_5$. It follows that
\begin{gather*}
\Omega^5_{\P^6}(2X)/\Omega^5_{\P^6,\sI}(2X) \cong N_{C\amalg C'/\P^6}(-1) = T_{\P^6}(-1)|_{C\amalg C'}/T_{C\amalg C'}(-1).
\end{gather*}
\minCDarrowwidth.6cm
We build a diagram (ignore for the moment the $\delta_i$ and the bottom line which will only be used in the next lemma):
\begin{gather}\label{17}
\begin{CD} 0 @>>> H^0\big(T_{\P^6}(-1)\big) @ >a>> H^0\big(C\amalg C', N_{C\amalg C'/\P^6}(-1)\big) @>b>> H^1\big(\Omega^5_{\P^6, \sI}(6)\big) @>>> 0 \\
@. @VV \delta_1 V @VV \delta_2 V @VV \delta_3 V \\
@. H^0\big(\sO_{\P^6}(2)\big) @>a'>> H^0\big(\sO_{\P^6}/\sI^2(2)\big) @>b'>> H^1\big(\sI^2\sO_{\P^6}(2)\big) @>>> 0
\end{CD}
\end{gather}
Recall that $\sI$ is the ideal of $C\amalg C'$, where $C$ is defined in homogeneous coordinates by $x_1=x_2=x_3=x_4=Q(x_5,x_6,x_7)=0$ and $C'$ is defined by $x_4=x_5=x_6=x_7=Q'(x_1,x_2,x_3)=0$. The kernel of $a'$ above consists of homogeneous polynomials of degree $2$ on $\P^6$ which vanish to degree~$2$ both on $C$ and on $C'$. I.e., $\ker(a') = \C x_4^2$. Also $b'$ is surjective because $H^1\big(\sO_{\P^6}(2)\big)=(0)$ and~$b$ is surjective because $H^1(\P^6, T_{\P^6}(-1))=(0)$.

To see that $a$ is injective, note the Euler sequence
\begin{gather*}
0 \to \sO_{\P^6}(-1) \to \bigoplus_0^6\C\frac{\partial}{\partial x_i} \to T_{\P^6}(-1) \to 0
\end{gather*}
yields
\begin{gather*}
H^0\big(\P^6, T_{\P^6}(-1)\big)= \bigoplus_0^6\C\frac{\partial}{\partial x_i}.
\end{gather*}
Recall $C\colon x_1=x_2=x_3=x_4=Q(x_5,x_6,x_7)=0$. The projection of $a$ onto $N_{C/\P^6}(-1)$ is given by $a(\partial/\partial x_i) = \partial/\partial x_i$, $i=1,2,3,4$ and $a(\partial/\partial x_j) = (\partial Q/\partial x_j)\partial/\partial Q$, $j=5,6,7$. This map is injective. Since $H^0\big( \Omega^5_{\P^6,\sI}(2X)\big)=(0)\subset \ker(a)$, the lemma follows.
\end{proof}
\begin{lem}\label{lem9}
The arrow
\begin{gather*}
H^1\big(\P^6, \Omega^5_{\P^6, \sI}(2X)\big) \xrightarrow{d} H^1\big(\P^6, (\sI(\sS)\cap \sI^2)\Omega^6_{\P^6}(3X)\big)
\end{gather*}
is injective.
\end{lem}
\begin{proof}[Proof of lemma]
The remaining arrows in~\eqref{17} are defined as follows. We identify
\begin{gather*}
H^1\big(T_{\P^6}(-1)\big)= \bigoplus_0^6 \frac{\partial}{\partial x_i}
\end{gather*}
as above, and define
\begin{gather*}
\delta_1\bigg(\frac{\partial}{\partial x_i}\bigg) =\frac{ \partial(\Psi_{3,1,3})}{\partial x_i} \in H^0\big(\sO_{\P^6}(2)\big).
\end{gather*}
Note that in fact $\delta_1(\partial/\partial x_i) \in H^0\big(\sI\sO_{\P^6}(2)\big)$ so we get a well-defined map
\begin{gather*}
T_{\P^6}(-1)|_\sC \to \sO_{\P^6}/\sI^2(2).
\end{gather*}
Tangent vectors along $\sC$ stabilize powers of $\sI$ so in fact the map factors through $N_{\sC/\sP^6}(-1) \to \sO_{\P^6}/\sI^2(2)$, defining $\delta_2$ in~\eqref{17}. Finally $\delta_3$ exists because the top row in~\eqref{17} is exact.

It is straightforward to check that the diagram
$$
\begin{CD}
H^1\big(\P^6, \Omega^5_{\P^6, \sI}(2X)\big)@>d>> H^1\big(\P^6, \big(\sI(\sS)\cap \sI^2\big)\Omega^6_{\P^6}(3X)\big) \\
@VV \delta_3 V @VVV \\
H^1\big(\sI^2\sO_{\P^6}(2)\big) @= H^1\big(\P^6, \sI^2\Omega^6_{\P^6}(3X)\big)
\end{CD}
$$
is commutative, so the lemma will follow if we show $\delta_3$ is injective. It is convenient to isolate this statement as a separate sublemma.
\begin{lem}[sublemma]\label{sublem10}
With reference to diagram~\eqref{17}, we have
\begin{gather*}
\operatorname{Image}(\delta_2)\cap \operatorname{Image}(a') = \operatorname{Image}(a'\circ \delta_1).
\end{gather*}
As a consequence, $\delta_3$ is injective.
\end{lem}
\begin{proof}
Fix $C\colon x_1=x_2=x_3=x_4=Q(x_5,x_6,x_7)=0$ to be one component of $\sC$. Let $I\supset \sI$ be the ideal of $C\subset \sC$. The normal bundle with a $-1$ twist is $N_{C/\P^6}(-1) \cong \sO_C^4\oplus \sO_C(1)$. Here we have to be careful because $\sO_C(1)$ is a line bundle of degree $2$ on the conic $C$ so $h^0(\sO_C(1))=3$ and the composition
\begin{gather*}
H^0\big(T_{\P^6}(-1)\big) \xrightarrow{a} H^0\big(\sC,N_{\sC/\P^6}(-1)\big) \xrightarrow{\text{proj}} H^0\big(C,N_{\sC/\P^6}(-1)\big)
\end{gather*}
is an isomorphism of vector spaces of dimension $7$.

Suppose now we have in the sublemma that $\delta_2(u)=a'(v)$. By the above, we can modify~$u$ (resp.~$v$) by some $a(w)$ (resp.~$\delta_1(w)$) so that $u$ is trivial on $C$, that is $a'(v) \in I^2(2)=(x_1,x_2,x_3,x_4)^2$. By assumption, $a'(v) \in \delta_2(H_{C'/\P^6}(-1))$ which means we can write (here $\Psi:=\Psi_{3,1,3}$)
\begin{gather}
a'(v) = c_4\frac{\partial\Psi}{\partial x_4}+\cdots +c_7\frac{\partial \Psi}{\partial x_7}
+c_8\frac{\partial Q'}{\partial x_1}\frac{\partial \Psi}{\partial Q'}+
c_9\frac{\partial Q'}{\partial x_2}\frac{\partial\Psi}{\partial Q'} + c_{10}\frac{\partial Q'}{\partial x_3}\frac{\partial \Psi}{\partial Q' }\nonumber
\\ \hphantom{a'(v)}
{}=c_4\frac{\partial \Psi}{\partial x_4}+\cdots +c_7\frac{\partial \Psi}{\partial x_7}
+ \bigg(c_8\frac{\partial Q'}{\partial x_1}+c_9\frac{\partial Q'}{\partial x_2}+c_{10}\frac{\partial Q'}{\partial x_3}\bigg)(x_4+x_5+x_6+x_7).\label{21}
\end{gather}
We want to show this expression cannot be a non-trivial quadric in $x_1,\dots,x_4$. The monomials which appear fall into $4$ classes:
\begin{gather*}
(123)^2, \qquad (123)(456),\qquad (4)(456),\qquad (567)^2.
\end{gather*}
(Here, e.g., $(123)^2$ refers to monomials $x_ix_j$ with $1\le i, j\le 3$.) The only terms in $(567)^2$ appear in~\eqref{21} with coefficient $c_4$ so we must have $c_4=0$. We have for $i=5,6,7$
\begin{gather*}
\frac{\partial \Psi}{\partial x_i} = \frac{\partial Q}{\partial x_i}(x_1+x_2+x_3+x_4)+Q'+x_4\frac{\partial \sA}{\partial x_i}.
\end{gather*}
The terms $Q'$ and $x_4\partial \sA/\partial x_i$ lie in $(1234)^2$. For the terms in~\eqref{21} not in $(1234)^2$ to cancel we must have for some constant $K$
\begin{gather}
c_5\frac{\partial Q}{\partial x_5} + \frac{c_6\partial Q}{\partial x_6} +\frac{c_7\partial Q}{\partial x_7}=K(x_5+x_6+x_7),\nonumber
\\
c_8\frac{\partial Q'}{\partial x_8}+c_9\frac{\partial Q'}{\partial x_9}+c_{10}\frac{\partial Q'}{\partial x_{10}}=-K(x_1+x_2+x_3).\label{22}
\end{gather}
It follows from~\eqref{21} and~\eqref{22} that $a'(v)$ in~\eqref{21} will involve a term $K(x_5+x_6+x_7)$ which does not cancel. This proves Lemma~\ref{sublem10}.
\end{proof}
The assertion (needed to prove Lemma~\ref{lem9}) that $\delta_3$ is injective is now just a diagram chase.
\end{proof}
\begin{lem}
\label{lem11}
$H^0\big(\P^6, I(\sS)\cap \sI^2\Omega^6_{\P^6}(3X_{3,1,3})\big) \cong \C$.
\end{lem}
\begin{proof}[Proof of lemma]
We will show that $H^0\big(\P^6, \sI^2\Omega^6_{\P^6}(3X_{3,1,3})\big) \cong \C$ and that the generating section also vanishes on points of $\sS$. In fact, $\sI^2\Omega^6_{\P^6}(3X_{3,1,3})\cong \sI^2_{\P^6}(2)$. Note $\sI^2=(I_C\cap I_{C'})^2$, so~we are looking for homogenous forms of degree $2$ which vanish to order 2 on $C$ and on $C'$. The only such forms lie in $\C\cdot x_4^2$. Finally, note that $x_4$ vanishes on $\sS$.
\end{proof}
Theorem \ref{mainthm} follows by combining Lemmas \ref{firstlem}, \ref{lem9}, and \ref{lem11}.
\end{proof}

\section{The Hodge structure in more detail}
In this section, we justify steps $2$ and $3$ from the introduction. A general reference is~\cite{PS}. 
\begin{prop}
With notation as above, we have $F^4H^5(Z,\C) = (0).$
\end{prop}
\begin{proof} Again from \cite[Chapter~II, Proposition~3.13]{D}, we have
\begin{gather*}
\H^1\big(W,\Omega^5_W(Z) \xrightarrow{d} \Omega^6_W(2Z)\big)\cong F^5H^6(W-Z,\C) \cong F^4H^5(Z,\C).
\end{gather*}
It will suffice to show
\begin{gather*}
H^0\big(W,\Omega^6_W(2Z)\big) = 0= H^1\big(W,\Omega^5_W(Z)\big).
\end{gather*}
$Z$ is birational with the cubic $X_{3,1,3}\subset \P^6$, and a section of $\Omega^6_W$ with only a pole of order $2$ along $Z$ would give rise to a section of $\Omega^6_{\P^6}$ with only a pole of order $2$ along $X_{3,1,3}$. Since $\Omega^6_{\P^6}\cong \sO_{\P^6}(-7)$ there are no such sections.

It remains to show $H^1\big(W, \Omega^5_W(Z)\big)=(0)$. By duality, it suffices to show $H^5\big(W, \Omega^1_W(-Z)\big)$ $=(0)$. Recall $W$ is obtained from $\P^6$ by blowing up $C\amalg C' \amalg \sS$. The exceptional divisors are~$E$,~$E'$, $\coprod_i \P^5 \subset W$. Define $\Xi=\Omega^1_W/\pi^*\Omega^1_{\P^6}$. We claim
\begin{gather*}
\Xi \cong \Omega^1_{E/C}\oplus \Omega^1_{E'/C'}\oplus \bigoplus_i\Omega^1_{\P^5}.
\end{gather*}

Indeed there is a natural surjective map just by restricting $1$-forms from $W$ to the exceptional divisors. Consider the local structure along $E$. Locally $C\colon z_1=\cdots =z_5=0$ in $\P^6$. Locally upstairs, we have $E\colon z_1=0$ and $\Omega^1_W=\pi^*\Omega^1_{\P^6}+\sum \sO_W{\rm d}(z_i/z_1)$. Note
\begin{gather*}
z_1{\rm d}(z_i/z_1)={\rm d}z_i-(z_i/z_1){\rm d}z_1 \in \pi^*\Omega^1_{\P^6}.
\end{gather*}
This implies that $\pi^*\Omega^1_{\P^6}\supset \sO_W(-E)\Omega^1_W$, so $\Omega^1_W|_E \surj \Xi$. Also the conormal bundle $N^\vee(E/W)= \sO_W(-E)/\sO_W(-2E)$ of $E\subset W$ lies in $\Omega^1_W|_E$. Similarly, the conormal bundle $N^\vee_{C/\P^6}$ lies in~$\Omega^1_{\P^6}|_C$, and the pullback $\pi^*N^\vee_{C/\P^6} \to N^\vee(E/W)$ is surjective. It follows that
\begin{gather*}
\Omega^1_E = \big(\Omega^1_W|_E\big)/N^\vee(E/W) \surj \Xi.
\end{gather*}
The proof of~\eqref{24} is now straightforward.

We have the exact sequence
\begin{gather} \label{24}
0 \to \pi^*\Omega^1_{\P^6}(-Z) \to \Omega^1_W(-Z) \to \Xi(-Z) \to 0.
\end{gather}
In the Picard group of $W$ we have $-Z=-\pi^*(X_{3,1,3})+2E+2E'+2\sum_{s\in \sS}\P^5_s$. We compute the direct images
\begin{gather}\label{25}
{\mathbf R}\pi_*\big(\pi^*\Omega^1_{\P^6}(-Z)\big) = \Omega^1_{\P^6}(-3)\otimes {\mathbf R}\pi_*\bigg(\sO_W\bigg(2E+2E'+2\sum_{s\in \sS}\P^5_s\bigg)\!\bigg).
\end{gather}
The calculus of intersection theory on blowups \cite[Section~II.7]{Ha}, yields $E\cdot E = [\sO_E(-1)]$, the dual of the tautological relatively ample bundle on the projective bundle $E$ over $C$ (and similarly for~$E'$ and the $\P^5_s$). Applying $\mathbf R\pi_*$ to the exact sequences
\begin{gather*}
0 \to \sO_W\Big(E+E'+\sum \P^5_s\Big) \to \sO_W\Big(2E+2E'+2\sum \P^5_s\Big)
\\ \phantom{0}
 \to \sO_E(-2)\oplus \sO_{E'}(-2) \oplus \bigoplus_{\sS}\sO_{\P^5}(-2) \to 0
\end{gather*}
and the similar sequence
\begin{gather*}
0 \to \sO_W \to \sO_W\Big(E+E'+\sum \P^5_s\Big) \to \sO_E(-1)\oplus \sO_{E'}(-1)+\sum \sO_{\P^5_s}(-1) \to 0
\end{gather*}
and recalling that for a projective space bundle $\P \xrightarrow{\pi}S$ with fibre $\P^m$ we have ${\mathbf R}\pi_*\sO_{\P}(-n)=(0)$ for $1\le n\le m$, we can see that
\begin{gather}\label{27}
{\mathbf R}\pi_*\bigg(\sO_W\bigg(2E+2E'+2\sum_{s\in \sS}\P^5_s\bigg)\!\bigg) \cong {\mathbf R}\pi_*(\sO_W).
\end{gather}
We want to use these results to prove $H^5\big(W, \Omega_W^1(-Z)\big) = (0)$. From~\eqref{24}, it will suffice to show
\begin{gather*}
H^5\big(W, \pi^*\Omega^1_{\P^6}(-Z)\big) = (0) = H^5(W,\Xi(-Z)).
\end{gather*}
From~\eqref{25} and~\eqref{27} we get
\begin{gather}\label{28}
H^5\big(W, \pi^*\Omega^1_{\P^6}(-Z)\big) =H^5\big(\P^6,\Omega^1_{\P^6}(-3)\otimes {\mathbf R}\pi_*\sO_W\big).
\end{gather}
We claim $\mathbf R\pi_*(\sO_W) \cong \sO_{\P^6}$, i.e., $R^i\pi_*\sO_W=(0)$ for $i\ge 1$. This can be checked locally on $\P^6$, so for notational simplicity we assume $W$ is $\P^6$ blown up along $C$. (Of course, $R^0\pi_*\sO_W=\sO_{\P^6}$. Also $R^i\pi_*(\sO_E)=(0)$ for $i\ge 1$.) The ideal sheaf $\sO_W(-E)$ is relatively ample (a basic fact about blowing up \cite[Section~II.7]{Ha}) so $R^i\pi_*(\sO_W(-nE))=(0)$ for $n\gg0$ and $i\ge 1$. We have $\big($note $E\cdot E = [\sO_E(-1)]\big)$
\begin{gather*}
R^i\pi_*\big(\sO_W(-nE)\big) \to R^i\pi_*\big(\sO_W(-(n-1)E)\big) \to R^i\pi_*\big(\sO_E(n-1)\big).
\end{gather*}
By descending induction we conclude $R^i\pi_*(\sO_W(-nE))=(0)$ for $n\ge 0$ and $i\ge 1$.

Finally, from~\eqref{28} we get
\begin{gather*}
H^5\big(W, \pi^*\Omega^1_{\P^6}(-Z)\big) = H^5\big(\P^6,\Omega^1_{\P^6}(-3)\big).
\end{gather*}
Dualizing and twisting the Euler sequence yields
\begin{gather*}
0 \to \Omega^1_{\P^6}(-3) \to \bigoplus_7 \sO_{\P^6}(-4) \to \sO_{\P^6}(-3) \to 0,
\end{gather*}
from which we conclude the desired vanishing for $H^5\big(W, \pi^*\Omega^1_{\P^6}(-Z)\big)$.

We need finally to check vanishing for $H^5(W, \Xi(-Z))$, with $\Xi$ as in~\eqref{24}. It suffices to show for example ${\mathbf R}\pi_*\big(\Omega^1_{E/C}(-2)\big)=(0)$. We have $0 \to \Omega^1_{E/C} \to \bigoplus_5 \sO_E(-1) \to \sO_E \to 0$, and the sheaves $\sO_E(-2)$ and $\sO_E(-3)$ are fibrewise $\sO_{\P^4}(-k), k=2, 3$ and have vanishing cohomology along the fibres in all degrees.
\end{proof}
\begin{prop}[verification of Step $3$ at the end of Section~\ref{S1}]
We have
\begin{gather*}
H^5(Z,\Q) \cong {\rm gr}^W_5H^5(X_{3,1,3},\Q).
\end{gather*}
\end{prop}
\begin{proof}
We consider the Leray spectral sequence (writing $\pi_Z=\pi|_Z\colon Z \to X_{3,1,3}$).
\begin{gather*}
E_2^{p,q} = H^p\big(X_{3,1,3},R^q\pi_{Z,*}\Q_Z\big) \Rightarrow H^{p+q}(Z,\Q).
\end{gather*}
The fibres of $\pi_Z$ are connected, so $\pi_*\Q_Z \cong \Q_{X_{3,1,3}}$ and $E_2^{5,0}=H^5(X_{3,1,3},\Q)$. We can stratify $X_{3,1,3}$ so the fibres consist either of one point or a (possibly singular) $3$-dimensional complex quadric (over $C\amalg C'$) or a $4$-dimensional complex quadric (over the isolated singular set $\sS$). The only odd $q\le 5$ representing a possible non-zero cohomology of a fibre is $H^3$ which can occur in isolated singular fibres for the $3$-dimensional quadric over $C\amalg C'$. But then $H^2\big(R^3\big)=(0)$. Thus, the only $E_2^{p,q}$ which contributes to $H^5(F)$ is $E_2^{5,0} = H^5(X_{3,1,3},\Q) \to H^5(F,\Q)$. This map is necessarily surjective. It factors through ${\rm gr}^W_5H^5(X_{3,1,3},\Q)$, proving Step~3.
\end{proof}

\appendix
\section{The symanzik polynomials}\label{app}
Let $(m,1,n)$ denote the $2$-loop graph obtained by subdividing the $3$ edges of the sunset graph $(1,1,1)$ into $m$ (resp.~1, resp.~$n$) edges. We associate edge variables $x_1,\dots,x_m$ (resp.~$x_{m+1}$, resp.~$x_{m+2},\dots,x_{m+n+1}$) to the subdivided edges in the evident way. The first Symanzik polynomial is the quadratic polynomial
\begin{gather*}
\phi_{m,1,n}(x_1,\dots,x_m,x_{m+1},x_{m+2},\dots,x_{m+n+1})= \sum x_ix_j
\end{gather*}
with the sum being taken over all pairs $x_i\neq x_j$ such that cutting edges $i$ and $j$ renders the graph a tree. One checks
\begin{gather*}
\phi_{m,1,n} = x_{m+1}(x_1+\cdots + x_m+x_{m+2}+\cdots + x_{m+n+1}) + \bigg(\sum_{i=1}^mx_i\bigg)\bigg(\sum_{j=m+2}^{m+n+1}x_{j}\bigg).
\end{gather*}

The (massless) second Symanzik polynomial $\psi_{m,1,n}$ has degree $3$ in the edge variables. The coefficients are quadratic in the momentum flows $w_\mu \in \C^D$ where $\mu$ runs over the vertices of the graph. A monomial $x_ix_jx_k$ appears in~$\psi_{m,1,n}$ if cutting the edges $i$, $j$, $k$ in $(m,1,n)$ reduces the graph to a $2$-tree. Here $2$-tree means a subgraph with no loops and exactly $2$ connected components. (Connected components may be isolated vertices.) The coefficient of~$x_ix_jx_k$ in~$\psi_{m,1,n}$ is computed as follows. Let $C$ be one of the two connected components of the cut graph $(m,1,n)-\{i,j,k\}$. The coefficient of $x_ix_jx_k$ in $\psi_{m,1,n}$ is
\begin{gather*}
c_{i,j,k} = \bigg(\sum_{\mu \in \text{vert}(C)} w_\mu\bigg)^2.
\end{gather*}
Here all vertex momenta are oriented to point inward, and the square refers to the quadratic form on $\C^D$. Because the sum of graph momenta vanishes, switching the choice of connected component $C$ does not change $c_{i,j,k}$.

A reminder: the geometric results in this paper are predicated on the coefficients $c_{i,j,k}$ being sufficiently generic. In particular, for the $(3,1,3)$ double box case, when $D=1$ the $c_{i,j,k}$ are never sufficiently generic. They are sufficiently generic if $D=4$ and the momentum vectors $w_\mu$ have generic entries.

Finally, the massive second Symanzik is given by
\begin{gather*}
\Psi_{m,1,n}:= \Big(\sum m_i^2x_i\Big)\phi_{m,1,n} + \psi_{m,1,n}.
\end{gather*}
\begin{prop}\label{2symm1n}
For generic values of masses $m_i$ and momentum vectors, the massive second Symanzik can be written
\begin{gather*}
\Psi_{m,1,n}(x_1,\dots,x_{m+n+1}) = Q(x_1,\dots,x_m)(x_{m+1}+\cdots + x_{m+n+1})
\\ \hphantom{\Psi_{m,1,n}(x_1,\dots,x_{m+n+1})=}
{} +
 Q'(x_{m+2},\dots,x_{m+n+1})(x_{1}+\cdots + x_{m+1}) + x_{m+1}\sA.
\end{gather*}
Here $Q$ and $Q'$ are homogeneous of degree $2$, and
\begin{gather}\label{11a}
\sA = (x_1,\dots,x_{m+1})\begin{pmatrix}a_{1,m+1} & a_{1,m+2}& \cdots & a_{1,m+n+1} \\
a_{2,m+1}& \cdots & \cdots & a_{2,m+n+1}\\
a_{m,m+1}& \vdots & \vdots & \vdots \\
0 & a_{m+1,2} & \cdots & a_{m+1,m+n+1} \end{pmatrix}\begin{pmatrix}x_{m+1}\\x_{m+2}\\ \vdots \\x_{m+n+1}\end{pmatrix}\!.
\end{gather}
\end{prop}
The proof of the proposition is straightforward using a computer. The author will leave a~program on his web page in the Publications folder.

Note that the terms in $\sA$ have generic coefficients. They are linear in $x_1,\dots,x_{m+1}$ and also in $x_{m+1},\dots,x_{m+n+1}$ and in $x_{m+1}$. The quadrics $Q$ and $Q'$ have generic coefficients.

\pdfbookmark[1]{References}{ref}
\LastPageEnding

\end{document}